\RequirePackage{fix-cm}
\documentclass{svjour3}                     
\smartqed  
\usepackage{graphicx}
\usepackage{enumitem}

 \usepackage{mathptmx}      
\usepackage{amssymb}
\usepackage{amsmath}

\usepackage{enumitem}
\usepackage{color}

\usepackage[colorinlistoftodos]{todonotes}

\renewcommand{\H}{\mathcal{H}}
\newcommand{\F}{\mathcal{F}}

\renewcommand{\P}{\operatorname{proj}}

\DeclareMathOperator{\dom}{dom}

\newtheorem{assumption}{Assumption}

\begin{document}

\title{Tikhonov-like regularization of  dynamical systems associated with nonexpansive operators defined in closed and convex sets
}

\titlerunning{Tikhonov-like regularization of  dynamical systems associated with nonexpansive operators}        

\author{Pedro P\'erez-Aros \and Emilio Vilches \thanks{E. Vilches was partially funded by CONICYT Chile under grants Fondecyt de Iniciaci\'on N$^{\circ}$ 11180098 and Fondecyt regular N$^{\circ}$ 1200283, P. P\'erez-Aros was partially supported by CONICYT Chile under grants Fondecyt regular N$^{\circ}$ 1190110 and Fondecyt regular N$^{\circ}$ 1200283 and Programa Regional Mathamsud 20-Math-08 CODE: MATH190003.}}



\institute{ 
              P. P\'erez-Aros \at
              Instituto de Ciencias de la Ingenier\'ia, Universidad de O'Higgins, Rancagua, Chile \\
              \email{pedro.perez@uoh.cl}           \and E. Vilches \at
              Instituto de Ciencias de la Ingenier\'ia, Universidad de O'Higgins, Rancagua, Chile\\
              \email{emilio.vilches@uoh.cl}
              }

\date{Received: date / Accepted: date}

\maketitle

\begin{abstract}
In this paper, we propose a new Tikhonov-like regularization for dynamical systems associated with non-expansive operators defined in closed and convex sets of a Hilbert space. We prove the well-posedness, the invariance, and the strong convergence of the proposed dynamical system to a fixed point of the non-expansive operator. We apply the obtained result to a dynamical system associated with the problem of finding the zeros of the sum of a cocoercive operator with the subdifferential of a convex function. 
 \and Cocoercive operator
\end{abstract}

\section{Introduction}\label{intro}

Let $D$ be a closed and convex set of a Hilbert space $\H$. In this paper, we are interested in the study of a  { Tikhonov-like regularization} of the following dynamical system
\begin{equation}\label{DyS}
\left\{
\begin{aligned}
-\dot{x}(t)&=x(t)-T(x(t))& \textrm{ a.e. } t\geq 0,\\
x(0)&=x_0\in D,
\end{aligned}
\right.
\end{equation}
where $T\colon D\to D$ is a non-expansive operator (see Definition \ref{Def-Nonexp} below) and $x_0\in D$.  The consideration of this dynamical system is motivated by the study of the set of fixed points of the operator $T$. Indeed, every equilibrium point of the dynamical system \eqref{DyS} is a fixed point of $T$.  The framework \eqref{DyS} includes the dynamical system proposed by Antipin in \cite{Antipin} 
 \begin{equation}\label{DyS-P}
\left\{
\begin{aligned}
-\dot{x}(t)&=x(t)-\operatorname{proj}_{C}\left(x(t)-\mu \nabla \varphi(x(t))\right)& \textrm{ a.e. } t\geq 0,\\
x(0)&=x_0,
\end{aligned}
\right.
\end{equation}
where $\varphi\colon \H \to \mathbb{R}$ is a convex $C^1$ function defined on a real Hilbert space $\H$, $C$ is a nonempty closed and convex set of $\H$, $x_0\in \H$ and  $\mu>0$. In this context, it was shown in \cite{Bolte2003} that the trajectory of \eqref{DyS-P} converges weakly to a minimizer of the optimization problem 
\begin{equation*}
\inf_{x\in C}\varphi(x),
\end{equation*}
provided that the latter is solvable. Latter, Abbas and Attouch \cite{Abbas2015} considered the following generalization of the dynamical system \eqref{DyS-P}
 \begin{equation}\label{DyS-Prox}
\left\{
\begin{aligned}
-\dot{x}(t)&=x(t)-\operatorname{prox}_{\mu \Phi}\left(x(t)-\mu B(x(t))\right)& \textrm{ a.e. } t\geq 0,\\
x(0)&=x_0,
\end{aligned}
\right.
\end{equation}
where $\Phi\colon \H \to \mathbb{R}\cup \{+\infty\}$ is a proper, convex and lower semicontinuous functions defined on a real Hilbert space $\H$, $B\colon \H \to \H$ is a cocoercive operator, $x_0\in \H$, $\mu >0$ and $\operatorname{prox}_{\mu \Phi}\colon \H \to \H$
\begin{equation*}
\operatorname{prox}_{\mu \Phi}(x):=\operatorname{argmin}_{y\in \H}\left\{ \Phi(y)+\frac{1}{2\mu}\Vert y-x\Vert^2\right\},
\end{equation*}
denotes the proximal point operator of $\Phi$.
Finally, in \cite{Bot2017}, the authors prove the weak convergence of the orbits of the dynamical system \eqref{DyS} to a fixed point for the operator $T\colon \H \to \H$, extending the results mentioned above.  

\noindent In this paper, we propose the following variant of the dynamical system \eqref{DyS}
\begin{equation}\label{DyS-Tik}
\left\{
\begin{aligned}
-\dot{x}(t)&=x(t)-T_t(x(t))+\varepsilon(t)(x(t)-y(t))& \textrm{ a.e. } t\geq 0,\\
x(0)&=x_0\in D,
\end{aligned}
\right.
\end{equation}
where $(T_t)_{t>0}$ is a family of  nonexpansive operator from $D$ into $D$ approaching a nonexpansive operator $T\colon D\to D$, as $t\to +\infty$, $y: [0,\infty) \to  D$ and $\varepsilon\colon [0,\infty)\to [0,\infty)$ are appropriate functions. The system \eqref{DyS-Tik} corresponds to a Tikhonov-like regularization of the dynamical system \eqref{DyS}. This kind of regularization has been considered by several authors (see, e.g., \cite{Attouch_Cominetti_1996,CPS2008}). In \cite{CPS2008}, the authors consider the system
\begin{equation}\label{MMO}
-\dot{x}(t)\in Ax(t)+\varepsilon(t)x(t),
\end{equation}
where $A$ is a maximal monotone operator defined on a Hilbert space and $\varepsilon(t)$ tends to $0$ as $t\to +\infty$ with $\int_0^{+\infty}\varepsilon(s)ds=+\infty$. They prove the strong convergence  towards the least-norm point in $A^{-1}(0)$ provided that the function $\varepsilon(t)$ has bounded variation.  This setting includes the operator $A=I-T$ when $T$ is defined in the whole Hilbert space $\H$. However, when the operator $T$ is defined only in a closed convex subset $D\subset \H$ with $D\neq \H$, the operator $I-T$ is not necessarily maximal monotone, and, thus, it is not clear  that the dynamical system \eqref{MMO} is well defined. 
 The main contributions of this paper is to prove that, under mild assumptions, the system \eqref{DyS-Tik} is well defined on $D$ (see Proposition \ref{Existencia-Tikhonov}) and that the orbits of \eqref{DyS-Tik} converge strongly to the point $\operatorname{proj}_{\operatorname{Fix}T}(y)$ provided the set $\operatorname{Fix}T$ is nonempty (see Theorem \ref{main}).  

We emphasize that our motivation to study the system \eqref{DyS-Tik}, which is governed by a family of time-dependent nonexpansive operators rather than a fixed one, is twofold. On the one hand, it enables accelerating the convergence towards a fixed point. On the other hand, it establishes a stability result with respect to perturbations on the initial operator, which may be helpful from the numerical point of view.

The paper is organized as follows. In Section \ref{Preliminaries}, we set the notation of the paper and prove preliminary results on non-expansive operators. In Section \ref{Nonexpansive}, we present the main properties of the dynamical system \eqref{DyS}. In Section \ref{Tikhonov-Sec}, we present the main result of the paper (see Theorem \ref{main}); namely, the strong convergence of the trajectories of the dynamical system \eqref{DyS-Tik} to a point in the set $\operatorname{Fix}T$. Then, we give some applications of the main result to the dynamical system \eqref{DyS-Prox}. The paper ends with conclusions and final remarks.

\section{Notation and preliminaries}\label{Preliminaries}
Let $\H$ be a Hilbert space endowed with a scalar product $\langle \cdot,\cdot\rangle$ and unit ball $\mathbb{B}$. 
\noindent Given a nonempty, closed and convex set $S\subset \H$, we define the distance function $d_S$ and the projection over $S$ as the maps
\begin{equation*}
\begin{aligned}
d_S(x)&:=\inf_{y\in S}\Vert y-x\Vert   & \textrm{ and } \quad \operatorname{proj}_S(x):=\{y\in \H\colon d_S(x)=\Vert x-y\Vert \}.
\end{aligned}
\end{equation*}
For $S$ as a above, it is not difficult to prove that the map $x\mapsto d_S^2(x)$ is differentiable with 
\begin{equation*}
\begin{aligned}
\nabla d_S^2(x) &=2\left(x-\operatorname{proj}_S(x)\right) & \textrm{ for all } x\in \H.
\end{aligned}
\end{equation*}  Moreover, the following inequality holds
\begin{equation}\label{inc-convexo}
\begin{aligned}
\left\langle x-\operatorname{proj}_S(x),y-\operatorname{proj}_S(x) \right\rangle \leq 0  & \textrm{ for all } y\in S.
\end{aligned}
\end{equation}
We refer to \cite{BC2017} for more details.

Let $\Phi\colon \H \to \mathbb{R}\cup \{+\infty\}$ be a proper, convex and lower semicontinuous function and $\mu>0$. The proximal point operator of $\Phi$ is defined as
\begin{equation}\label{proximal-point}
\operatorname{prox}_{\mu \Phi}(x):=\operatorname{argmin}_{y\in \H}\left\{ \Phi(y)+\frac{1}{2\mu}\Vert y-x\Vert^2\right\}.
\end{equation}
The proximal point operator is everywhere well defined and the map $x\mapsto \operatorname{prox}_{\mu \Phi}(x)$ is Lipschitz of constant $1$ (see \cite[Section~12.4]{BC2017}).  Moreover, when $\Phi$ is the indicator function of a closed and convex set $C$, then the proximal point operator coincides with the projection operator over $C$, that is,  $\operatorname{prox}_{\mu \Phi}(x)=\operatorname{proj}_C(x)$.

The proximal point operator plays a fundamental role in optimization theory. Indeed, this operator is the basis of several optimization algorithms (see, e.g., \cite{BC2017}). Moreover, it is well known (see, e.g.,  \cite[Proposition~12.29]{BC2017}) that the set of  fixed points of this operator coincides with the set of solution of the problem 
$$
\inf_{x\in \H} \Phi(x).
$$
\noindent The following definitions will be used throughout the paper.
\begin{definition}\label{Def-Nonexp}
 An operator $T: D \subset \H \to  \H$ is called
\begin{enumerate}
\item $\beta$-cocoercive on $D$, if 
	\begin{equation*}
	\begin{aligned}
	\langle T(x) -T(y) ,x-y \rangle &\geq \beta \| T(x)- T(y)\|^2 & \textrm{ for all } x,y\in D.
	\end{aligned}
	\end{equation*}
\item Non-expansive on $D$ if 
	\begin{equation*}
	\begin{aligned}
	\| T(x) -T(y)\| &\leq \| x- y\|  & \textrm{ for all } x,y\in D.
	\end{aligned}
	\end{equation*}
\item Firmly non-expansive on $D$ if 
	\begin{equation*}
	\begin{aligned}
	\| T(x) -T(y)\|^2 + \|( \operatorname{Id}-T)(x) - ( \operatorname{Id}-T)(y) \|^2  &\leq \| x- y\|^2& \textrm{ for all } x,y\in D.
	\end{aligned}
	\end{equation*}
	
\end{enumerate}
\end{definition}
It is important to mention that if $D$ is closed and convex and $T\colon D \to D$ is non-expansive, then the set $\operatorname{Fix}T$ is closed and convex (see, e.g.,  \cite[Corollary~4.24]{BC2017}). Moreover, if, in addition, $D$ is bounded, the \emph{Browder-G\"{o}hde-Kirk's Theorem} (see, e.g., \cite[Theorem 4.29]{BC2017}) asserts that the set $\operatorname{Fix}T$ is nonempty. On the other hand,  if $T$ is $\alpha$-Lipschitz with $\alpha\in [0,1)$, then $\operatorname{Fix}T$ is a singleton.

Let us consider a non-expansive operator $T\colon D\to D$ and define the operator $G\colon D\to \H$ given by
\begin{equation}\label{operator-G}
G(x)=x-T(x).
\end{equation}
It is clear that the set of fixed point of $T$ coincides with the set of zeros of $G$. Moreover, according to \cite[Proposition~4.4]{BC2017}, if the operator $T$ is non-expansive, then $T$ is monotone. The following lemma gives the existence of approximate zeros of $G$.
\begin{lemma}\label{Proposition2.1}
Assume that $T\colon D\to D$ is non-expansive and fix $y\in D$. Then, for every $\epsilon>0$ there exists a unique $x_\epsilon^y \in D$ such that
	\begin{align}\label{unique_solution}
	\epsilon x_\epsilon^y + G(x_\epsilon^y) =\epsilon y.
	\end{align}
	\end{lemma}

\begin{proof}
	Since the operator $T$ is non-expansive on $D$, we can apply \cite[Proposition 4.30]{BC2017}, to obtain that  for all $\eta \in (0,1)$ there exists a unique $x_\eta^y \in D$ such that 
	\begin{align*}
	x_\eta^y = \eta y + (1- \eta )  T(x_\eta^y).
	\end{align*}
	In particular, taking $\eta =\frac{\epsilon}{1+\epsilon }$, we obtain the existence of a unique 
	$x^y_\epsilon \in D$ such that 
	\begin{align*}
	x^y_\epsilon = \frac{\epsilon}{1+\epsilon } y + \left(1- \frac{\epsilon}{1+\epsilon } \right)  T(x^y_\epsilon),
	\end{align*}
	which implies the result. \qed
\end{proof}

Now, let us define $\F: (0,+\infty) \times D \to D$, given by
\begin{align}\label{defintionF}
\F(\epsilon,y)=x_\epsilon^y ,
\end{align} 
where $x_\epsilon^y$ is the unique solution of \eqref{unique_solution} given by Lemma \ref{Proposition2.1}.

The next results give us some properties of the trajectory $x_\epsilon^y$
\begin{lemma}\label{lemmacomb}
	Consider the function $\F$ defined in \eqref{defintionF}. Then,
	\begin{enumerate}[label=\roman*)]
		\item\label{lemmacomba} For all $\epsilon>0$ the function $\F(\epsilon,\cdot)$ is firmly nonexpansive on $D$.
		\item If $\operatorname{Fix}T=\emptyset$, then for all $y \in D$, $\lim\limits_{\epsilon \to 0^+ } \| \F (\epsilon,y) \| =+\infty$.
		\item For all $\epsilon>0$, and all $x^\ast \in  \operatorname{Fix}T$ 
		\begin{align*}
			\| y - \F(\epsilon, y) \|^2 +\|\F(\epsilon, y)-x^\ast \|^2 \leq \|y - x^\ast\|^2. 
		\end{align*}
		\item  If $\operatorname{Fix}T\neq \emptyset$, then
		\begin{align*}
			\lim\limits_{\epsilon \to 0^+} \F(\epsilon,y)= \P_{\operatorname{Fix}T}(y).
		\end{align*}
		\item For all $y \in D$, the function $\epsilon \to  \|y- \F(\epsilon,y)\|$ is decreasing.
		\item For all $y \in D$, the function $\F(\cdot,y)$ is continuous.
	\end{enumerate}
\end{lemma}
\begin{proof}
	See \cite[Proposition 4.30]{BC2017}. \qed
\end{proof}

The following result is fundamental to establish continuity properties for the map $\varepsilon \mapsto \mathcal{F}(\varepsilon,y)$ for $y\in D$ fixed.
\begin{lemma}\label{identidad_F}
	Consider $\mu >\lambda>0 $. Then for every $y\in D$
	\begin{align}\label{transformacion}
	\F(\lambda, y )=\F \left(\mu, \frac{\lambda }{\mu} y + \left(1 - \frac{\lambda}{ \mu }\right) \F(\lambda ,y)  \right).
	\end{align} 
\end{lemma}
\begin{proof}
	We observe that
	\begin{equation*}
	 \F(\lambda ,y)\in D, \textrm{ and } \frac{\lambda }{\mu} y + \left(1 - \frac{\lambda}{ \mu }\right) \F(\lambda ,y) \in D,
	 \end{equation*} 
	 thus \eqref{transformacion} is well-defined. \newline
	To end the proof, it is enough to verify that $\F(\lambda ,y)$ satisfies \eqref{unique_solution} with $$\epsilon = \mu \text{ and  } z=\frac{\lambda }{\mu} y + \left(1 - \frac{\lambda}{ \mu }\right) \F(\lambda ,y).$$ Indeed, 
	\begin{equation*}
	\begin{aligned}
	 \mu \F(\lambda ,y) + G(\F(\lambda ,y)) &=  \mu \F(\lambda ,y) +\lambda \F(\lambda ,y)+  G(\F(\lambda ,y))  - \lambda\F(\lambda ,y) 
	 \\&=  \mu \F(\lambda ,y) +  \lambda y - \lambda \F(\lambda ,y)\\ &= \mu \left( \frac{\lambda}{ \mu }y +  (1 -\frac{\lambda}{ \mu })\F(\lambda ,y)    \right),
	\end{aligned}
	\end{equation*}
	which ends the proof. \qed
	\end{proof}

The following proposition establishes the continuity and differentiability almost everywhere of the map $\epsilon \mapsto \F(\epsilon,x)$ for $x\in D$ fixed.
\begin{proposition}\label{Prop-AC}
	For every $\epsilon_1,\epsilon_2>0 $ and $x\in D$
	\begin{align}
	\| \F(\epsilon_2,x) - \F(\epsilon_1,x)\| \leq \frac{\vert \epsilon_2- \epsilon_1 \vert}{ \min\{\epsilon_1,\epsilon_2\}  } \| x - \F(\min\{\epsilon_1,\epsilon_2\},x)\|.
	\end{align} 
	Consequently, for every $x\in D$ the function $\F(\cdot,x)$ is locally Lipschitz and for all $x\in D$ and a.e. $t>0$
	\begin{align*}
	 \left\Vert  \frac{ d }{ d t  } \F(\cdot, x)(t)\right\Vert  \leq \frac{1 }{ t} \| x-\F(t,x)\|. 
	\end{align*}
	\end{proposition}

\begin{proof} Fix $x\in D$ and assume that $\epsilon_2 >\epsilon_1$. Then, according to Lemma \ref{identidad_F},
	\begin{align*}
	\| \F({\epsilon_2},x) -\F({\epsilon_1},x)\| = \left\Vert \F({\epsilon_2},x   )- \F\left({\epsilon_2},\frac{\epsilon_1}{ \epsilon_2 }x   +(1-  \frac{\epsilon_1}{ \epsilon_2 }) \F(\epsilon_2,x) \right) \right\Vert. 
	\end{align*}
	Next, due to Lemma \ref{lemmacomb} \ref{lemmacomba}, we know that $\F(\epsilon_2,\cdot)$ is nonexpansive. Thus,
		\begin{align*}
	\| \F({\epsilon_2},x) -\F({\epsilon_1},x)\| &\leq \left( 1 - \frac{\epsilon_1}{ \epsilon_2 }\right)\, \|x- \F(\epsilon_2,x)\|=\frac{ \left(\epsilon_2 -\epsilon_1\right)}{ \epsilon_2} \|x- \F(\epsilon_2,x)\|. 
	\end{align*}
Repeating the argument for $\epsilon_1 >\epsilon_2$, we obtain the result. \qed
	\end{proof}

\section{Dynamical systems associated with nonexpansive operators}\label{Nonexpansive}
In this section, let us formally establish some convergence result to  the following dynamical system
\begin{equation}\label{PD}
\left\{
\begin{aligned}
-\dot{x}(t)&=x(t)-T(x(t))& \textrm{ a.e. } t\geq 0,\\
x(0)&=x_0\in D,
\end{aligned}
\right.
\end{equation}
where $T\colon D \to D$ is a nonexpansive operator defined over a closed and convex set $D\subset \H$. The dynamical system \eqref{PD} was considered in \cite{Bot2017} for a non-expansive operator $T\colon \H \to \H$ defined in the whole space. We extend this result for a merely closed and convex set $D\subset \H$.  The results established here will be used to compare with the Tikhonov-like regularization proposed in this article.

The following proposition establishes the well-posedness and invariance of \eqref{PD}.
\begin{proposition}\label{FijoD}
 If $T\colon D \to D$ is a nonexpansive operator, then the dynamical system \eqref{PD} admits a unique solution $x\in \operatorname{AC}_{\operatorname{loc}}\left(\mathbb{R}_+;\H\right)$. Moreover, this solution satisfies $x(t)\in D$ for all $t\geq 0$.
\end{proposition}
\begin{proof}
	See the proof of Proposition \ref{Existencia-Tikhonov} below. \qed
\end{proof}
The following proposition establishes convergence properties of the dynamical system \eqref{PD}. Its proof follows from Proposition \ref{FijoD} and the ideas of \cite[Theorem~6]{Bot2017}, where it is established for the case $D=\H$.
\begin{proposition}\label{Propo-Bot}
	Let $T\colon D \to D$ be a nonexpansive operator such that $\operatorname{Fix}T\neq \emptyset$. Let $x(\cdot)$ be the unique solution of (\ref{PD}). Then the following assertions hold:
	\begin{enumerate}
		\item[(i)] the trajectory $x$ is bounded and $\int_0^{+\infty}\Vert \dot{x}(t)\Vert^2 dt<+\infty$;
		\item[(ii)] $\lim_{t\to +\infty}\left(T(x(t))-x(t)\right)=0$ and for all $t>0$
		$$
		\Vert x(t)-Tx(t)\Vert \leq \frac{1}{\sqrt{t}}\operatorname{dist}\left(x_0,\operatorname{Fix}T\right);
		$$
		\item[(iii)] $\lim_{t\to +\infty}\dot{x}(t)=0$;
		\item[(iv)] $x(t)$ converges weakly to a point in $\operatorname{Fix}T$, as $t\to +\infty$.
	\end{enumerate}
	Moreover, if $T$ is $\alpha$-Lipschitz with $\alpha\in [0,1)$, then the unique fixed point $x^*$ of $T$ is globally exponentially stable, that is,
	\begin{equation*}
	\begin{aligned}
	\Vert x(t)-x^*\Vert &\leq e^{-(1-\alpha)t}\Vert x_0-x^{*}\Vert  &\textrm{ for all }t\geq 0.
	\end{aligned}
	\end{equation*}
\end{proposition}

\section{Tikhonov-like regularization}\label{Tikhonov-Sec}

In this section, we study the \emph{Tikhonov-like regularization} for the projected dynamical system \eqref{PD}.  We propose the following dynamical system:
\begin{equation}\label{SDP}
\left\{
\begin{aligned}
-\dot{x}(t)&=x(t)-T_t(x(t))+\varepsilon(t)(x(t)-y(t)) & \textrm{ a.e. } t\geq 0,\\
x(0)&=x_0\in D,
\end{aligned}
\right.
\end{equation}
where $(T_t)_{t>0}$ is a family of nonexpansive operators from $D$ into $D$, $y$ and  $\varepsilon$ are  function satisfying Assumption \ref{hipo1} below. The following proposition establishes, under general assumptions, the well-posedness and the invariance for the dynamical system \eqref{SDP}.
\begin{proposition}\label{Existencia-Tikhonov}
	Assume that $(T_t)_{t>0}$ is a family of  nonexpansive operators from $D$ into $D$. If the maps $t\mapsto \varepsilon(t)$,  $t\mapsto y(t)$ and $t\mapsto T_t(x)$ are locally integrable  for all $x\in D$.   Then, the dynamical system \eqref{SDP} admits a unique solution $x\in \operatorname{AC}_{\operatorname{loc}}\left(\mathbb{R}_+;\H\right)$. Moreover, this solution satisfies $x(t)\in D$ for all $t\geq 0$. 
\end{proposition}
\begin{proof}
	Let us consider the dynamical system
	\begin{equation}\label{SDP-2}
	\left\{
	\begin{aligned}
	-\dot{x}(t)&=\operatorname{proj}_D\left(x(t)\right)-T_t(\operatorname{proj}_D\left(x(t)\right))+\varepsilon(t)(\operatorname{proj}_D\left(x(t)\right)-y(t)) & \textrm{ a.e. } t\geq 0,\\
	x(0)&=x_0\in D,
	\end{aligned}
	\right.
	\end{equation}
	According to the classical Cauchy-Lipschitz theorem (see, e.g., \cite[Theorem 30.9, p. 819]{Schechter_1997}), the dynamical system (\ref{SDP-2}) has a unique solution $x\in\operatorname{AC}_{\operatorname{loc}}\left(\mathbb{R}_+;\H\right)$. We aim to prove that $x(t)\in D$ for all $t\geq 0$. To do that, we define the function $\psi(t):=\frac{1}{2}d_D^2(x(t))$. This function is absolutely continuous and for a.e. $t\geq 0$
	\begin{equation*}
	\begin{aligned}
	\dot{\psi}(t)&= \langle x(t) -\P_D(x(t)),\dot{x}(t)\rangle\\
	&=\langle x(t) -\P_D(x(t)) ,T_t(\P_D(x(t)))-\P_D(x(t)) \rangle\\
	&+\varepsilon(t)  \langle x(t)-\P_D(x(t)),y(t)-\P_D(x(t)) \rangle \\
	&\leq 0,
	\end{aligned}
	\end{equation*}
	where we have used the inequality \eqref{inc-convexo}. Thus, since $\psi(0)=0$, it follows that $\psi(t)=0$ for all $t\geq 0$. Therefore, $\operatorname{proj}_D\left(x(t)\right)=x(t)$ for all $t\geq 0$, which proves that $x$ is the unique solution of the dynamical system \eqref{SDP}. \qed
\end{proof}

In order to provide asymptotic convergence of the solution of the dynamical system \eqref{SDP}, we consider the following assumptions: 

\begin{assumption}\label{hipo1} Let $(T_t)_{t>0}$ be a family of nonexpansive operators from $D$ into $D$,  $\varepsilon\colon \mathbb{R}_+\to \mathbb{R}_+$ be a positive function, and $y\colon \mathbb{R}_+\to D$ satisfying 
\begin{enumerate}[label=(\alph*)]
\item\label{hipo1a} $\varepsilon$ is absolutely continuous, nonincreasing with  $\lim\limits_{t\to +\infty}\varepsilon(t)=0$ and $\int_0^{+\infty}\varepsilon(s)ds=+\infty$;
\item\label{hipo1b} the map $t\mapsto y(t)$ is locally absolutely continuous, $y(t) \to y \in D$ and $\| \dot{ y}(t)\| \to 0$ as $t\to \infty$.
\item\label{hipo1c} the map $t\mapsto T_t(x)$ belongs to $L^1_{\operatorname{loc}}\left(\mathbb{R}_+,\H\right)$ for all $x\in D$. 
\item \label{Ass:f} $\lim\limits_{t \to +\infty} \psi (t)/ \varepsilon(t)=0$, where 
\begin{align}\label{def:psi}
\psi(t):= 2\| y(t) -y\| +       w(t,  \F(\varepsilon(t),y)) + \|\dot{ y}(t)\| - \frac{\dot{\varepsilon}(t)}{\varepsilon(t)} \left( 2 \| y(t) -y\| + d_{\operatorname{Fix}T}(y) \right),
\end{align}
with  $\F ( \epsilon,y)$ the unique solution of \eqref{unique_solution} given by Lemma \ref{Proposition2.1} associated to $T$, and 
\begin{align}\label{Ass:e}
w(t,x):=\| T_t(x) - T(x) \|  \text{ for all } t \geq 0\text{ and  all } x\in D.
\end{align}

\end{enumerate}

\end{assumption}
The following proposition shows that Assumption \ref{hipo1} can be obtained through a suitable time scale reparametrization.
	\begin{proposition}\label{prop}
		Let  $(T_t)_{t>0}$ be a family of nonexpansive operators from $D$ into $D$,  $\varepsilon\colon \mathbb{R}_+\to \mathbb{R}_+$ be a positive function, and $y\colon \mathbb{R}_+\to D$ satisfying statements  \ref{hipo1a}, \ref{hipo1b} and \ref{hipo1c}  of Assumption \ref{hipo1}. Suppose that  $\lim_{ t\to \infty} {\dot{\varepsilon}(t)}/{\varepsilon^2(t)} =0$, $\lim_{ t\to \infty} t \cdot\left(  \| y(t) -y\| +\| \dot{y}(t)\| \right) =0$, and  for every compact set $K \subset D$, $\lim_{ t\to \infty} t\cdot  \sup_{x\in K}  \| T_t(x) -T(x)\|=0$. Then, the family of nonexpansive operators $\hat T_t(x):= T_{1/\epsilon(t)}(x) $ and the function $\hat{ y}(t)=y(1/\epsilon(t)) $ satisfies statement \ref{Ass:f} of Assumption \ref{hipo1}.
		\end{proposition} 
\begin{proof}
It is easy to check that $  \|\hat{ y}(t)-y\|/\epsilon(t) \to 0$, as $t\to \infty$. Moreover, $$\frac{ d}{dt} \hat{y} (t) = -\dot{y}(1/\epsilon(t)) \frac{\dot{\epsilon}(t)}{\epsilon^2(t)},$$ which shows that 
$$\frac{1}{\varepsilon(t)}\dfrac{d}{dt} \hat{y} (t) \to 0, \textrm{ as } t\to \infty.$$ Finally, since $\F(\epsilon(\cdot),y)$ is continuous and $\F(\epsilon(t),y)$ converges as $t\to \infty$ (see Lemma \ref{lemmacomb}), we have that the set $K:=\{\F(\epsilon(t),y) : t\geq 0 \}$ is compact on $\H$. Then
\begin{align*}
\frac{\|\hat T_t(\F(\epsilon(t),y)) - T(\F(\epsilon(t),y)))  \| }{\epsilon(t) }  \leq \frac{1}{\epsilon(t)} \cdot \sup_{x\in K}    \| T_{1/\epsilon(t)}(x) -T(x)\| \to 0, \textrm{ as } t\to +\infty.
\end{align*}
Therefore, $\lim\limits_{t\to \infty} \psi(t)/\epsilon(t)=0$. Thus, we have verified the statement \ref{Ass:f} of Assumption \ref{hipo1} .
	\end{proof}

\begin{remark}
	We observe that in Proposition \ref{prop},  for instance, the function  $\varepsilon(t)={(1+t)^{-\beta}}$ with $\beta\in (0,1)$ satisfies  $\lim_{t\to +\infty}{\dot{\varepsilon}(t)}/{\varepsilon^2(t)}=0$. Moreover, since the function $w$, defined in \eqref{Ass:e}, is only measurable in $t$, the limit in statement \ref{Ass:f} could be taken in a \emph{essentially sense}, that is $\text{ess}\lim\limits_{t \to +\infty} \psi (t)/ \varepsilon(t)=0$, which means for all $\gamma >0$, there exists $t_\gamma \in (0,+\infty)$ such that
	\begin{align*}
	\frac{\psi (t)}{\varepsilon(t)} \leq \gamma, \text{ for almost all }t \geq t_\gamma.
	\end{align*}
	\end{remark}

The next theorem is the main result of the paper and establishes, under mild assumptions, that the point $\operatorname{proj}_{\operatorname{Fix}T}(y)$ is globally asymptotically stable for the dynamical system \eqref{SDP}, that is, 
\begin{equation}\label{GAS}
x(t)\to {\operatorname{proj}}_{\operatorname{Fix}T}(y) \textrm{ as }  t\to +\infty.
\end{equation}

We underline that the importance of this result is that the obtained convergence in \eqref{GAS} is strong, in contrast with the dynamical system considered in Proposition  \ref{Propo-Bot}.
\begin{theorem}\label{main}
Suppose that $T\colon D \to D$ is a non-expansive operator with $\operatorname{Fix}T\neq \emptyset$. If Assumption \ref{hipo1} holds, then the unique solution $x(t)$ of \eqref{SDP} converges strongly to $\operatorname{proj}_{\operatorname{Fix}T}(y)$, as $t\to +\infty$. 
\end{theorem}
Before giving the proof, it is important to emphasize the advantages of Theorem \ref{main} compared with Proposition \ref{Propo-Bot}. First, the trajectories of \eqref{SDP} are always defined in the set $D$. Thus, $T$ needs only to be defined in this set. Second, from the numerical point of view, the cost of discretizing \eqref{SDP} is relatively the same as discretizing the dynamical system \eqref{PD}. Third, the convergence obtained in Theorem \ref{main} is strong, while the convergence in Proposition \ref{Propo-Bot} is weak. Finally, all the trajectories of \eqref{SDP} (regardless of the starting point) converge strongly to the point $\operatorname{proj}_{\operatorname{Fix}T}(y)$.
\begin{proof} Define the operators $$G_t(x):=x-T_t(x),\; G(x):=x-T(x).$$ Since $T_t(\cdot)$ and $T$ are nonexpansive on $D$, it is clear that $G_t(\cdot)$ and $G$ are  monotone operators on $D$ (see, e.g.,  \cite[Example 20.7]{BC2017}). \newline
\noindent Consider the functions $z(t):=\F(\varepsilon(t),y(t))$ and  $$\theta(t):=\frac{1}{2}\Vert x(t)-z(t)\Vert^2.$$ According to Proposition  \ref{Prop-AC}, $\theta$ is locally absolutely continuous and for almost all $t\geq 0$, we have
\begin{equation*}
\begin{aligned}
\dot{\theta}(t)=& \left\langle x(t)-z(t),\dot{x}(t)- \frac{ d  }{dt  }  z(t) \right\rangle \\
=& \left\langle x(t)-z(t),-G_t(x(t)) +\epsilon(t)y(t) -\epsilon(t)x(t) \right\rangle  -\left\langle x(t)-z(t), \frac{ d  }{dt  }  z(t)     \right\rangle \\ 
=&   \left\langle x(t)-z(t),-G_t(x(t))+G_t(z(t))\right\rangle   +\left\langle x(t)-z(t), -G_t(z(t))+G_t(\F(\varepsilon(t),y)\right\rangle \\ & +\left\langle  x(t)-z(t),  -G_t(\F(\varepsilon(t),y)) +G(z(t))\right\rangle +\epsilon(t) \left\langle x(t)-z(t),z(t) -x(t)\right\rangle \\ & -\left\langle x(t)-z(t),  \frac{ d  }{dt  }  z(t) \right\rangle,
\end{aligned}
\end{equation*}
where we have used the definition of  $z(t)$ (see  Equation \eqref{unique_solution}). Now, let us upper estimate the right-hand side of the above inequality. First, since $G_t(\cdot)$ is monotone  on $D$ and $x(t)\in D$ for all $t\geq 0$, the following inequality holds   
$$\left\langle x(t)-z(t),-G_t(x(t))+G_t(z(t))\right\rangle\leq 0.$$
Moreover,  on the one hand, using that $T(t,\cdot)$ and $\F(\epsilon,\cdot)$ are nonexpansive (see Lemma \ref{lemmacomb})
\begin{align*}
\left|  \left\langle x(t)-z(t), G_t(z(t))- G_t(\F(\varepsilon(t),y))\right\rangle \right| &\leq \sqrt{2\theta(t)}\| G_t(z(t))- G_t(\F(\varepsilon(t),y))\|  \\
&\leq 2\sqrt{2\theta(t)} \| y(t) -y\|.
\end{align*}
On the other hand,  using \eqref{Ass:e}, we get
\begin{align*}
 \left| \left\langle x(t)-z(t), G_t(\F(\varepsilon(t),y))  - G(\F(\varepsilon(t),y))\right\rangle  \right|& \leq \sqrt{2\theta(t)} \| G_t(\F(\varepsilon(t),y))  -G(\F(\varepsilon(t),y))\| \\
 &= \sqrt{2\theta(t)}  \| T_t(\F(\varepsilon(t),y)) - T(\F(\varepsilon(t),y))\|\\
 & = {\sqrt{2\theta(t)}}  w(t,  \F(\varepsilon(t),y)),
 \end{align*}
 Now,  using  Lebourg's mean value (see, e.g., \cite[Theorem 2.3.7]{Clarke1990}), we have the following estimation 
 \begin{align*}
 \left\| \frac{ d  }{dt  }  z(t) \right\| \leq - \frac{\dot{\varepsilon}(t)}{\varepsilon(t)}   \Vert y(t)-z(t)\Vert  + \| \dot{y}(t)\|, \text{ a.e. } t\geq 0, 
 \end{align*}
 where we have used Proposition \ref{Prop-AC} and assertions i) and v) of Lemma \ref{lemmacomb}. 
 
 Finally,
\begin{align*}
 \left\langle x(t)-z(t),  \frac{ d  }{dt  }  z(t) \right\rangle &\leq  \sqrt{2\theta(t)}  \left\| \frac{ d  }{dt  }  z(t) \right\|  \\ &\leq  \sqrt{2\theta(t)} \left( - \frac{\dot{\varepsilon}(t)}{\varepsilon(t)}   \Vert y(t)-z(t)\Vert  + \| \dot{y}(t)\| \right) , \\
&\leq \sqrt{2\theta(t)} \left(  - \frac{\dot{\varepsilon}(t)}{\varepsilon(t)} \left( 2 \| y(t) -y\| + \|y-  \operatorname{proj}_{\operatorname{Fix}T}(y)\| \right)  + \| \dot{y}(t)\|\right).
\end{align*}
\noindent Thus, for a.e. $t\geq 0$
\begin{equation*}
\begin{aligned}
\dot{\theta}(t)+2\varepsilon(t)\theta(t)\leq \psi(t) \sqrt{2\theta(t)},
\end{aligned}
\end{equation*}
where $\psi$ was defined in \eqref{def:psi}. Hence, the function $\varphi(t)=\sqrt{2\theta(t)}$ satisfies 
\begin{equation*}
\begin{aligned}
\dot{\varphi}(t)+\varepsilon(t)\varphi(t)\leq \psi(t) \textrm{ for a.e. }  t\geq 0.
\end{aligned}
\end{equation*}
Therefore, for all $t\geq 0$
\begin{equation*}
\begin{aligned}
\varphi(t)\leq \exp\left(-\int_0^t\varepsilon(\tau)d\tau\right)\left(\varphi(0)+ \int_0^t  \psi(s)   \exp\left(\int_0^s\varepsilon(\tau)d\tau\right)  ds\right)
\end{aligned}
\end{equation*}
Moreover, since Assumption \ref{hipo1} holds, the right-hand side in the last inequality goes to zero. Indeed, let us denote $r(t):= \exp\left(-\int_0^t\varepsilon(\tau)d\tau\right)\left( \int_0^t  \psi(s)   \exp\left(\int_0^s\varepsilon(\tau)d\tau\right)  ds\right)$. Consider $\gamma>0$, then there exists $t_{\gamma}>0$ such that $\psi(s)/\varepsilon(s)\leq \gamma$, for almost all $s>t_{\gamma}$. Hence, for all $t\geq t_{\gamma}$
\begin{equation*}
\begin{aligned}
r(t)&= \exp\left(-\int_0^t\varepsilon(\tau)d\tau\right)\left( \int_{t_{\gamma}}^t  \psi(s)   \exp\left(\int_0^s\varepsilon(\tau)d\tau\right)  ds+\int_0^{t_{\gamma}}  \psi(s)   \exp\left(\int_0^s\varepsilon(\tau)d\tau\right)  ds\right)\\
&\leq  \exp\left(-\int_0^t\varepsilon(\tau)d\tau\right)\left(\gamma  \int_{t_{\gamma}}^t  \varepsilon(s)   \exp\left(\int_0^s\varepsilon(\tau)d\tau\right)  ds+\int_0^{t_{\gamma}}  \psi(s)   \exp\left(\int_0^s\varepsilon(\tau)d\tau\right)  ds\right).
\end{aligned}
\end{equation*}
Moreover, as $t\to +\infty$,
\begin{equation*}
\begin{aligned}
\exp\left(-\int_0^t\varepsilon(\tau)d\tau\right)\left(\int_0^{t_{\gamma}}  \psi(s)   \exp\left(\int_0^s\varepsilon(\tau)d\tau\right)  ds\right) \to 0,\\
\exp\left(-\int_0^t\varepsilon(\tau)d\tau\right)\left(\gamma  \int_{t_{\gamma}}^t  \varepsilon(s)   \exp\left(\int_0^s\varepsilon(\tau)d\tau\right)  ds\right) \to \gamma.
\end{aligned}
\end{equation*}
Therefore, for all $\gamma>0$, $\limsup_{t\to +\infty}r(t)\leq \gamma$, which proves that $\lim_{t\to +\infty}\varphi(t)=0$.
%
%
%
Finally, according to Lemma \ref{lemmacomb} i) and iv), 
\begin{align*}
\limsup_{t\to +\infty}\Vert x(t)-\operatorname{proj}_{\operatorname{Fix}T}(y)\Vert \leq &\limsup_{t\to +\infty} \Vert x(t)-\F(\varepsilon(t),y(t))\Vert \\&   +  \limsup_{t\to +\infty}\Vert \F(\varepsilon(t),y(t))-\F(\varepsilon(t),y)\Vert  \\ & + \limsup_{t\to +\infty} \Vert \F(\varepsilon(t),y)-\operatorname{proj}_{\operatorname{Fix}T}(y)\Vert \\
\leq &\limsup_{t\to +\infty}  \Vert x(t)-\F(\varepsilon(t),y(t))\Vert    +  \limsup_{t\to +\infty} \Vert y(t) - y\Vert  \\ & + \limsup_{t\to +\infty} \Vert \F(\varepsilon(t),y)-\operatorname{proj}_{\operatorname{Fix}T}(y)\Vert,
\end{align*}
which ends the proof. \qed
\end{proof}

\section{Applications}

In this section, we present two  applications of Theorem \ref{main} to the dynamical systems \eqref{DyS-Prox} and \eqref{DyS-P}, respectively.

\noindent Consider a proper, convex and lower semicontinuous function $\Phi\colon \H \to \mathbb{R} \cup \{+\infty\}$ and its proximal point operator defined in \eqref{proximal-point}. Assume that $B\colon D\subset \overline{\operatorname{dom}\Phi}\to \H$ is a $\beta$-cocoercive operator, that is,
\begin{equation*}
\begin{aligned}
\left\langle Bx-By,x-y\right\rangle & \geq \beta \Vert Bx-By\Vert^2 & \textrm{ for all } x,y\in \overline{\operatorname{dom}\Phi}.
\end{aligned}
\end{equation*} 
The following result is a well known property.
\begin{proposition}\label{proposition6}
If $\mu\in (0,2\beta)$, then the operator $T:=x\mapsto \operatorname{prox}_{\mu \Phi}\left(x-\mu Bx\right)$ is non-expansive. Moreover, for all $x,y\in D$
\begin{align}\label{ineq}
\Vert Tx-Ty\Vert^2+\mu(2\beta-\mu)\Vert Bx-By\Vert^2\leq \Vert x-y\Vert^2.
\end{align}
\end{proposition}
Thus, as a  consequence of Theorem \ref{main}, we have the following theorem which extends the results from \cite{Abbas2015}.
\begin{theorem}\label{main2}
Let $ \Phi_t, \Phi : \H \to \mathbb{R} \cup\{ +\infty\}$ be a family of proper, convex, and lsc functions, $D$ be a nonempty, closed and convex set such that\{ $D\supseteq  \overline{\dom \Phi},\,    \overline{\dom \Phi_t}$, and  let $B_t, B \colon D \to \H$ be a family of $\beta$-cocoercive operators. Assume that  $y\in D$, $\mu\in (0,2\beta)$ and 
\begin{enumerate}
	\item[(i)] The function $\epsilon$ is absolutely continuous, nonincreasing with  $\lim\limits_{t\to +\infty}\varepsilon(t)=0$, $\int_0^{+\infty}\varepsilon(s)ds=+\infty$ and $\lim\limits_{t\to +\infty}{ \dot{\epsilon}(t)}/{\epsilon^2(t)}=0$.
	
	\item[(ii)] The map $t\mapsto \Phi_t(x)$ is measurable for all $x\in \H$ and there exists $\kappa\colon \mathbb{R}_+\to \mathbb{R}_+$ such that 
	\begin{align}\label{eq20}
	\Vert \operatorname{prox}_{\mu \Phi_t}(z)-\operatorname{prox}_{\mu \Phi}(z)\Vert \leq \kappa(t) (\Vert z\Vert +c_{\Phi})^{p} \textrm{ for all } (t,z)\in \mathbb{R}_+\times \H,
	\end{align}
	where  $c_{\Phi}$ and $p$ are positive constants and $\lim_{t\to +\infty}\kappa(t)/\varepsilon(t)=0$.
	
	\item[(iii)] For every compact set $K \subset D$,
		\begin{align}\label{eq}
		\lim\limits_{t\to +\infty} \frac{1}{\varepsilon(t)}\sup \limits_{x\in K} \|B_t(x) - B(x)\| =0.
	\end{align}
	\item[(iv)] { For every $x\in D$, the map $t\to   \operatorname{prox}_{\mu \Phi}\left(x-\mu B_tx\right) $ is measurable.}
\end{enumerate}

Let $x\colon [0,+\infty) \to \H$ be the unique solution of
\begin{equation}\label{dynfinal}
\left\{
\begin{aligned}
-\dot{x}(t)&=x(t)-\operatorname{prox}_{\mu \Phi_t}\left(x(t)-\mu B_tx(t)\right) +\varepsilon(t)(x(t)-y),\\
x(0)&=x_0 \in D.
\end{aligned}
\right.
\end{equation}
Then $x(t)$ converges strongly to $\operatorname{proj}_{\operatorname{zer}(\partial \Phi +B)}(y)$, as $t\to +\infty$, provided that $\operatorname{zer}(\partial \Phi +B)\neq \emptyset$.
\end{theorem}
\begin{proof}
	Let us define  $T_t(x):= \operatorname{prox}_{\mu \Phi}\left(x-\mu B_tx\right)$ and    $T(x):= \operatorname{prox}_{\mu \Phi}\left(x-\mu Bx\right)$,{ by Proposition \ref{proposition6}, they are nonexpansive operators from $D$ to $D$}. We will verify the hypotheses $(a)$-$(d)$ of Assumption \ref{hipo1}  to apply Theorem \ref{main}. It is clear that $(a)$ and  $(b)$  hold. Moreover,  the local integrability in $(c)$ follows from \eqref{ineq}. Let us verify $(d)$. Indeed, define $z_t:=\F(\epsilon(t),y))$ for $t>0$. We have
	\begin{equation*}
	\begin{aligned}
	w(t,\F(\epsilon(t),y)) &= \| T_t( \F(\epsilon(t),y))  -T(\F(\epsilon(t),y))\| \\
	&\leq   \Vert   \operatorname{prox}_{\mu \Phi_t}\left(z_t-\mu B_t z_t\right) - \operatorname{prox}_{\mu \Phi}\left(z_t-\mu B_t z_t\right)\Vert + 	\mu \|	B_t(z_t) - B(z_t)\|\\
	&\leq \kappa(t) (\Vert z_t-\mu B_t z_t \Vert +c_{\Phi})^{p}+\mu	\|	B_t(z_t) - B(z_t)\| \\
	&\leq  \kappa(t) \sup_{x\in K}(\Vert x-\mu B_t x \Vert +c_{\Phi})^{p}+\mu\sup_{x\in K}	\|B_t(x) - B(x)\|,
	\end{aligned}
	\end{equation*}
	where $K:=\{\F(\epsilon(t),y) : t\geq 0 \}$ is compact on $\H$. Hence,  for all $t>0$,
	\begin{equation*}
	\frac{\psi(t)}{\varepsilon(t)} \leq 
	\frac{\kappa(t)}{\varepsilon(t)} \sup_{x\in K}(\Vert x-\mu B_t x \Vert +c_{\Phi})^{p}+\frac{\mu}{\varepsilon(t)}\sup_{x\in K}	\|B_t(x) - B(x)\|-\frac{\dot{\varepsilon}(t)}{\varepsilon^2(t)}d_{\operatorname{Fix}T}(t),
	\end{equation*}
	which, by virtue of $(i)$-$(iii)$, implies that $\lim_{t\to +\infty} \psi(t)/\varepsilon(t)=0$. Finally,  the result follows from Theorem \ref{main}.
	
\end{proof}
\{\begin{remark}
		
It is worth to emphasize that in the last theorem the operator $B_t$ must to be defined only in closed and convex sets containing of $\overline{\operatorname{dom}\Phi}$ and not in all the space as in \cite{Bot2017}. Moreover, the trajectory converges to a point in $\dom \Phi$, but the initial point could be outside of the set. 
\end{remark}

In the rest of this section, we present a Tikhonov-like regularization for the dynamical system \eqref{DyS-P}.  To this end, we recall the following result from \cite{Perez-Vilches2019-2}, which is a Baillon-Haddad theorem for convex functions defined in open and convex sets (see also \cite[Theorem~3.3]{BC2010} for the twice continuously differentiable case).
\begin{proposition}\label{Haddad1}
 Let $U$ be a nonempty open convex subset of $\H$, let $f\colon U\to \mathbb{R}$ be convex and Fr\'echet differentiable on $U$, and let $\beta>0$. Then $\nabla f$ is $1/\beta$-Lipschitz continuous if and only if it is $\beta$-cocoercive.
\end{proposition}
The importance of Proposition \ref{Haddad1} is that it provides a class of cocoercive operators which are not necessarily defined in the whole space.

\noindent Recall that the set of optimal solutions for the problem
\begin{equation}\label{optim-prob}
\min_{x\in C}\varphi (x),
\end{equation}
is $\operatorname{zer}(\nabla \varphi +N(C;\cdot ))=\{z\in C \colon 0\in  \nabla \varphi(z)+N(C;z)\}$.   Moreover, under mild assumptions, the trajectories of \eqref{DyS-P} approaches weakly to an optimal solution of the optimization problem \eqref{optim-prob} (see Proposition \ref{Propo-Bot} and \cite{Antipin}). The following theorem provides a Tikhonov-like regularization for the dynamical system \eqref{DyS-P} proposed by Antipin \cite{Antipin}.
\begin{theorem}\label{main3}
Let $(C_t)_{t>0}, C$ and $D$ be nonempty, closed and convex sets such that {$  C_t,\, C \subseteq D$, and let $\varphi_t, \varphi$  be a family of $1/\beta$-Lipschitz and convex functions defined on a convex and open set containing $D$.} Assume that  $y\in D$, $\mu\in (0,2\beta)$ and 
\begin{enumerate}
	\item[(i)] The function $\epsilon$ is absolutely continuous, nonincreasing with  $\lim\limits_{t\to +\infty}\varepsilon(t)=0$, $\int_0^{+\infty}\varepsilon(s)ds=+\infty$ and $\lim\limits_{t\to +\infty} { \dot{\epsilon}(t)}/{\epsilon^2(t)}=0$.
	
	\item[(ii)] The sets $(C_t)_{t>0}$ and $C$ are uniformly bounded and $$\lim_{t\to +\infty} \frac{\operatorname{Haus}^{1/2}(C_t,C)}{\varepsilon(t)}=0,$$
	where $\operatorname{Haus}(C_t,C)$ denotes the Hausdorff distance between $C_t$ and $C$, given by 
	\begin{align*}
	\operatorname{Haus}(C_t,C)=\max\left\{ \sup_{ x\in C_t} d(x,C),  \sup_{ x\in C} d(x,C_t)   \right\}.
	\end{align*}
	
	\item[(iii)] For every compact set $K \subset D$,
		\begin{align}\label{eq}
		\lim\limits_{t\to +\infty} \frac{1}{\varepsilon(t)}\sup \limits_{x\in K} \| \nabla \varphi_t(x) - \nabla \varphi(x)\| =0.
	\end{align}
	\item[(iv)] For every $x\in D$, the map $t\to   \varphi_t(x) $ is measurable.
\end{enumerate}
Let $x\colon [0,+\infty) \to \H$ be the unique solution of
\begin{equation}\label{dynfinal}
\left\{
\begin{aligned}
-\dot{x}(t)&=x(t)-\operatorname{proj}_{C_t}\left(x(t)-\mu \nabla \varphi_t(x(t))\right) +\varepsilon(t)(x(t)-y),\\
x(0)&=x_0 \in D.
\end{aligned}
\right.
\end{equation}
Then $x(t)$ converges strongly to $\operatorname{proj}_{\operatorname{zer}(\nabla \varphi +N(C;\cdot ))}(y)$, as $t\to +\infty$, provided that $\operatorname{zer}(\nabla \varphi +N(C;\cdot ))\neq \emptyset$.
\end{theorem}
\begin{proof}
Let us consider $\Phi:=\delta_C$, $\Phi_t:=\delta_{C_t}$, the indicator function of $C$ and $C_t$, respectively, and the operators $B_t:=\nabla \varphi_t$  and $B:=\nabla \varphi$. By Proposition \ref{Haddad1}, we have that $B_t$ and $B$ are $\beta$-cocoercive on $D$. Then, in order to apply Theorem \ref{main2}, we have to check that \eqref{eq20} holds (the other assumptions can be checked easily). Indeed, by  \cite[Inequality (2.17) in Lemma p. 362]{Moreau1977} we have that
$$
\Vert \operatorname{proj}_{C_t}(z)-\operatorname{proj}_{C}(z)\Vert^2\ \leq 2(d_{C_t}(z)+d_C(z))\operatorname{Haus}(C_t,C) \textrm{ for all } z\in \H.
$$
Furthermore, without loss of generality, we can assume that $ c:=\sup_{ t\geq 0} \operatorname{Haus}(C_t,C) <+\infty$,   then    fixing $z_0\in C$, we have
\begin{align*}
d_{C_t}(z)+d_C(z) \leq  \operatorname{Haus}(C_t,C) +2\| z-z_0\|\leq  \operatorname{Haus}(C_t,C) +2\| z\|+2\|z_0\|.
\end{align*}
 Hence, \eqref{eq20} holds with $p=1/2$, $c_\Phi=\|z_0\|+\frac{1}{2}c$,    and $\kappa(t):= 2 \operatorname{Haus}^{1/2}(C_t,C)$.  Therefore, by virtue of Theorem \ref{main2}, we get the result.
\end{proof}

\section{Conclusions and final remarks}
In this paper, we propose a new Tikhonov-like regularization for dynamical systems associated with time-dependent non-expansive operators defined in closed and convex sets (possibly not  the whole space $\H$). Our main contribution is to prove well-posedness, invariance, and strong convergence of the proposed dynamical system to a specific point in the set $\operatorname{Fix}T$, provided provided that the latter is not empty. This result extends known results in the literature and, in particular, proposes a dynamical system whose solution is defined in the domain of the non-expansive operator $T$ and converging strongly to a fixed point. Moreover, as an application of our results, we propose Tikhonov-like regularization for two dynamical systems arising in the study of optimization problems. We expect that the discretization of the dynamical systems in this paper will help the design of algorithms to find fixed points of non-expansive operators.



\bibliography{references}
    \bibliographystyle{spmpsci}      


\end{document}